\documentclass[12pt]{article}
\usepackage{graphics}
\usepackage{graphicx}

\newtheorem{theorem}{Theorem}

\newtheorem{lemma}[theorem]{Lemma}

\newenvironment{proof}%
    {{\sc Proof.}}%
  {{\sc q.e.d.} \\}
\newenvironment{remark}%
    {\noindent {\bf Remark.}}%
  {{\sc } \\}
    {\noindent {\bf Definition  }}%
  {{\sc } \\}%
\newenvironment{proofclaim}%
    {{\sc Proof of Claim.}}%
  {{\sc q.e.d.(Claim)} \\}

\begin{document}

\title{Local uniformization  and free boundary regularity
of minimal singular surfaces } \maketitle

\maketitle
\begin{center}

\begin{tabular}{ccc}
Chikako Mese
 &  &
Sumio Yamada
\\ Department of Mathematics & &Mathematical Institute\\ Johns Hopkins
University & & Tohoku University\\ {\tt cmese@math.jhu.edu} & & {\tt
yamada@math.tohoku.ac.jp}\\
\end{tabular}
\end{center}
\noindent $^*${\scriptsize Research for the first author partially
supported by grant NSF DMS-0604930 and the second by JSPS Grant-in-Aid for Scientific
Research No.17740030.}
\renewcommand{\abstractname}{Abstract}
\begin{abstract}
In continuing the study of harmonic mapping from 2-dimensional Riemannian simplicial complexes 
in order to construct minimal surfaces with singularity, we obtain an a-priori regularity result concerning the  
real analyticity of the free boundary curve.  The free boundary is the singular set along which three disk-type minimal surfaces meet.  Here the configuration of the singular minimal surface is obtained by a minimization of
a weighted energy functional, in the spirit of J.Douglas' approach to the Plateau Problem.
Using the free boundary regularity of the harmonic map, we construct a local uniformization of the singular surface as a parameterization of a neighborhood of a point on the free boundary by the singular tangent cone.
In addition, applications of the local uniformization are discussed in relation to H.Lewy's real analytic extension
of minimal surfaces.
\end{abstract}

\section{Introduction}

The classical Plateau Problem is the problem of finding a surface minimizing the area amongst all surfaces which are images of a  map from a disk and spanning a given Jordan curve.  We can formulate this problem more precisely as follows.  Let $\triangle$ be the unit disk in ${\bf R}^2$.      The area of a  map $\alpha:\triangle \rightarrow {\bf R}^n $ is 
\[
A(\alpha)=\int_{\triangle} \sqrt{ \left| \frac{\partial \alpha}{\partial x} \right|^2 + \left| \frac{\partial \alpha}{\partial y} \right|^2-2  \frac{\partial \alpha}{\partial x} \cdot \frac{\partial \alpha}{\partial y}}
\ dxdy.
\] 
{\bf The classical Plateau Problem.}  {\it Given a  Jordan curve $\Upsilon$, let 
\[
{\cal F}=\{\alpha:\overline{\triangle} \rightarrow {\bf R}^n: \alpha \in W^{1,2}(\triangle) \cap C^0(\overline{\triangle}) \mbox{ and } \alpha\big|_{\partial \triangle}:\partial \triangle \rightarrow \Upsilon \mbox{ is a homeomorphism}\}.
\]
Find $\alpha^* \in {\cal F}$ so that $A(\alpha^*) \leq A(\alpha)$ for all $\alpha \in {\cal F}$.} \\
\\
The classical Plateau Problem can be solved by finding a map $\alpha^*$ which is conformal and harmonic.   The study of classical minimal surfaces via harmonic maps is now well developed, starting with the solution of this problem by J.Douglas~\cite{douglas} in the 1930's. 

Our interest in this paper is a singular version of this problem.  
In~\cite{mese-yamada}, we investigated
properties of 1-dimensional (2-dimensional resp.) subsets of a Euclidean space which minimize length (area resp.)  amongst all the sets sharing a certain boundary set.  Due to the shape of  the boundary set, the minimizer is not a curve (surface resp.) but a union of curves (surfaces resp.).
A variational formulation of 
this problem was presented in~\cite{mese-yamada} 
as a minimization problem for 
a certain energy functional of maps from a suitable polyhedral domain (or a  finite complex).   
Continuing the study of the two dimensional version of the problem, we refer to a singular surface as a union of the images of three maps, each from a disk joined along a curve $\gamma$ and bounding a graph $\Gamma$ as pictured below.  The boundary set $\Gamma$ in this picture is the graph
consisting of three edges meeting at two vertices, and the curve $\gamma$ is the set where three disk-type surfaces
are meeting.  In the singular Plateau Problem, we look for a singular surface which minimizes area amongst all other singular surfaces bounding the given graph $\Gamma$.  \\

\begin{figure}[ht]
\begin{center}
\includegraphics[scale=0.35]{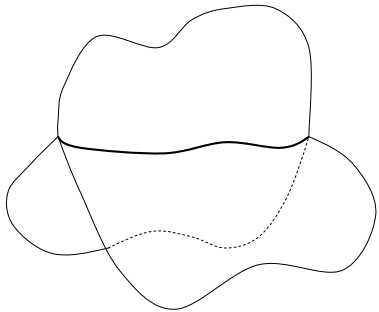}
\end{center}
\caption{Singular Surface}
\label{fig:surface}
\end{figure}

 In order to capture the area-minimizer as an energy-minimizer in our singular setting where  several (three) energy functionals are being minimized concurrently, we introduced in~\cite{mese-yamada} a weighted $l^2$-energy.  As the name suggests,  this functional is a quadratic combination of Dirichlet energies with respect to some weights.  Its minimizer in the spaces of maps, weights, and conformal structures/glueing maps  is shown to be the area minimizer.  Furthermore, the minimizing map is  harmonic and weakly conformal.  
The extension of the classical harmonic map theory to singular spaces was initiated by work of 
Gromov-Schoen \cite{gromov-schoen} 
and also Korevaar-Schoen \cite{korevaar-schoen1}, \cite{korevaar-schoen2} 
where singular target  spaces are considered.  The study of harmonic maps from a polyhedral space is considered in J. Chen \cite{chen}, 
Eells-Fuglede \cite{eells-fuglede}, 
Fuglede \cite{fuglede1} \cite{fuglede2} \cite{fuglede3} \cite{fuglede4} and  
Daskalopoulos-Mese \cite{daskal-mese1} \cite{daskal-mese2} \cite{daskal-mese3} \cite{daskal-mese4}.  The use of the energy functional in singular minimal surface theory is a natural application of the singular harmonic map theory.

The best known singular surface minimizing the area  is probably the $(M, 0, \delta)$-minimal sets in ${\bf R}^3$ (cf. \cite{almgren}, \cite{T}).
J.Taylor~\cite{T} showed that in the interior of a $(M, 0, \delta)$-minimal set, there are only two types of
singularities.   One is the {\bf Y}-type, characterized by  three minimal surfaces meeting at $120^\circ$ degrees
along a 
$C^{1,\alpha}$ curve (which we will call a free boundary).  The other is the {\bf T}-type, characterized by four of those free boundary curves
meeting at a vertex, whose tangent cone is the cone over the regular tetrahedron and 
each free boundary is $C^{1, \alpha}$ up to the vertex point ($0<\alpha< 1$). 
In~\cite{mese-yamada}, we attempted to construct the singular configuration of three surfaces meeting along a
free boundary as an image of a harmonic map from a simplicial complex.  The key technical aspect 
in this approach is that there is an infinite dimensional moduli space $\cal P$ of conformal structures defined on the 
domain complex which is homeomorphic to three standard 2-simplices  glued along a common 1-simplex.  Here the infinite
dimensionality is caused by the different ways of glueing the three copies of the 2-simplices.   
The fact that the free boundary is real analytic in the $(M,0,\delta)$-minimal setting  suggests that
the gluing maps one needs to consider for a minimizing element is smooth.

The first result (Theorem~\ref{kreg}) in this paper is that when the conformal structure of a minimizing element of our variational problem is induced by
a $C^{2, \alpha}$-smooth gluing map of 2-simplices, then the minimal singular surface has real analytic free boundary.
It is an a-priori estimate, not unlike the result of Kinderlehrer-Nirenberg-Spruck~\cite{KNS}, in the sense that given a minimal configuration of a variational
problem with a sufficient regularity, then the actual regularity is much better.   Recall that  \cite{KNS} proves that  the  area minimizing singular surfaces of {\bf Y}-type in ${\bf R}^3$ which was shown to be $C^{1,\alpha}$ by \cite{T} is in fact real analytic.
The phenomena is accounted by the general principle of Lopatinski-Shapiro's, where an elliptic  system of partial differential equations satisfies a coercive boundary condition.  

In order to set up our variational formulation according to this principle, we define  a weighted $l^1$-functional of the Dirichlet energies.
We show (Lemma~\ref{k-min/c-min}) that minimizing the $l^1$-energy is equivalent to minimizing the $l^2$-energy.    The $l^1$-energy minimizer satisfies the matching and balancing conditions (explained more precisely below, see (\ref{mf}) and (\ref{bf})) at the free boundary. This can be recognized as the coercive boundary condition in the sense of Lopatinski-Shapiro.

Having established the real analyticity of the free boundary, we geometrically characterize the
image of the minimizing map.  Where the map is non-degenerate, namely away from the set $S$ of isolated branch points,  the image is  locally either  an embedded minimal surface or a union of three minimal surfaces
meeting along a real analytic curve at $\pi/3$ angle (Theorem~\ref{three}).  Note that this is consistent with the description of 
$(M, 0, \delta)$-minimal surface by~\cite{T}.   Near a point on the free boundary parameterized by the non-degenerate map, the real analyticity of the three surface meeting along the real analytic free boundary is used to show that each of the induced conformal structures on the three minimal surfaces can be locally uniformized by a half disk, with a matching 
parameterization of the free boundary.  This in turn says that there is a  conformal (and harmonic) map 
locally defined on the unit ball (denoted $Y_0$) about the vertex of the tangent cone at a point of the singular point.   Thus we establish a local uniformization
(Theorem~\ref{isothermal})
of the minimal singular surface away from the branch points.  

Note that if the minimal surface is non-singular, this is nothing but the existence of an isothermal coordinate system.  
The important difference in the two characterizations of
minimal surfaces via a conformal harmonic map is that in the
classical Plateau Problem case,  the parameterization is of the entire disk-type
minimal surface, while in the singular case we only assert the
existence of a neighborhood of a {\bf Y}-type singular point conformally
equivalent to $Y_0$.    This is because any two disk-type
surfaces are conformally equivalent by the uniformization theorem
which has no analog in the singular setting. 
From this viewpoint, the key observation contained  in this paper
and \cite{mese-yamada} is that {\it locally there is only one conformal
structure} $Y_0$ to the singular surfaces while {\it globally there are
many}, parameterized by the infinite dimensional space $\cal P$. We further note that the  conformal structure 
$Y_0$ is conformally equivalent to the trivial element $X_{{\rm Id}}$ (where the glueing via the identity map) of  $\cal P$.

The existence of an isothermal coordinate system for the minimizer of our variational problem is closely related to
the real analytic extension of each surface across the free boundary.  The real analytic extension of a minimal surface
across a real analytic boundary curve is a celebrated result by H.Lewy~\cite{lewy} in the 50's.  The key to his proof is 
Schwarz' reflection principle of holomorphic function.  Having the isothermal coordinate system around a {\bf Y}-type singular point
of the image of our minimizer,  we construct the extension (Theorem~\ref{reflection}) of one surface  via a weighted averaging of the  harmonic maps parameterizing the three surfaces.  We call this construction multi-sheeted reflection.  The averaging works due to the important fact that one can regard these parameterizations as a {\it multi-valued harmonic function} defined on a half-disk and hence taking linear combinations  preserves harmonicity.  

Furthermore when the ambient space is ${\bf R}^3$, the matching and the balancing conditions of the three minimal surfaces along a real analytic free boundary 
give an overdetermined boundary condition. Indeed one can construct the configuration from a single minimal surface with a real analytic 
boundary with the use of Bj\"{o}ling's
method.  In this way, we provide an alternative expression (Theorem~\ref{lewy'}) for the extended surface across the boundary curve.
 
{\it Acknowledgement.  } We would like to thank Joel Spruck for an
explanation of the paper~\cite{KNS}. The second author thanks the Japan-U.S. Mathematics Institute
at Johns Hopkins University for the hospitality during his visits.

\section{Variational formulation}

In order to formulate the singular Plateau Problem precisely, we will need  three copies of $\triangle$ which will be  labeled  $\triangle_i$ for $i=1,2,3$.   Let 
\[
B_i=\partial \triangle_i \cap \{y \geq 0\} \mbox{ and } A_i=\partial \triangle_i \cap \{y \leq 0\}.
\]
We say $\Phi=(\phi_1,\phi_2,\phi_3)$ is a gluing map if $\phi_i:\partial \triangle_i \rightarrow \partial \triangle_i$ is a homeomorphism satisfying the three point condition:

\[
\phi_i(-1,0)=(-1,0), \phi_i(1,0)=(1,0) \mbox{ and } \phi_i(0,1)=(0,1).
\]
We denote the  set of all gluing maps by $\cal P$.

   We say  $\Phi \in {\cal P}$ is $K$-quasisymmetric ($C^k$, $C^{k,\alpha}$ resp.) if $\phi_i$ is $K$-quasisymmetric ($C^k(\partial \triangle_i)$, $C^{k,\alpha}(\partial \triangle_i)$ resp.) for all $i=1,2,3$.      For $\Phi \in {\cal P} \cap  C^1$, we say $\Phi$ is non-singular if 
\[
\frac{d \phi_i}{d\theta}(p) \neq 0
\]
for all $p \in \partial \triangle_i$ where $\theta$ is the arc-length parameter of $\partial \triangle_i$.

An extension of a gluing map $\Phi \in {\cal P}$ is  $\Phi'=(\phi_1',\phi_2',\phi_3')$ where $\phi_i':\overline{\triangle_i} \rightarrow \overline{\triangle_i}$ is a diffeomorphism so that $\phi_i'\big|_{\partial \triangle_i} =\phi_i$  for $i=1,2,3$.   We say an extension $\Phi'$ of $\Phi$ is a $D$-quasiconformal  ($C^k$, $C^{k,\mu}$ resp.)   if $\phi_i'$ is  $D$-quasiconformal ($C^k(\overline{\triangle_i})$, $C^{k,\mu}(\overline{\triangle_i})$ resp.)   for all $i=1,2,3$.   

 An element $\Phi \in {\cal P}$ defines a natural quotient space $X_{\Phi}$ of $\cup \triangle_i$ as follows.  First, we define an equivalent relation of $\cup A_i$ by setting $s \sim_{\Phi} t$ if $s \in A_i$, $t \in A_j$ and $\phi_i^{-1}(s)=\phi_j^{-1}(t)$.  We can then set
\begin{equation} \label{space}
X_{\Phi}=\cup \triangle_i/\sim_{\Phi}.
\end{equation}
The boundary $\partial X_{\Phi}$ of $X_{\Phi}$ is simply  $\cup B_i$.

We note here a difference of treatment concerning $X_{\Phi}$ from our previous paper~\cite{mese-yamada}, where
$X_{\Phi}$ was defined using $\Phi = (\phi_1, \phi_2, \phi_3)$ where $\phi_i$'s are quasiconformal diffeomorphisms
of the disk.  It turns out that $X_\Phi$ is completely determined by the parameterizations of $\partial \triangle_i$, as the different (quasiconformal) extensions inside $\triangle_i$ with the same quasisymmetric boundary map are regarded equivalent.   

Let $\Gamma$ be a graph embedded in ${\bf R}^n$ consisting of three arcs $\Gamma_i$ ($i=1,2,3$) sharing common end points $q_-$ and $q_+$.  We will refer to such $\Gamma$ as the fixed boundary of the Plateau Problem.   We say $\Psi=(\psi_1,\psi_2,\psi_3)$ is a uniform parameterization of $\Gamma$ if  $\psi_i:B_i \rightarrow {\bf R}^n$ is a constant speed parameterization of $\Gamma_i$ so that $q_-=\psi_i(-1,0)$ and $q_+=\psi_i(1,0)$.   

Fix a uniform parameterization $\Psi$ of $\Gamma$ .  For each  $\Phi \in {\cal P}$, we define ${\cal F}_{\Psi}(\Phi)$  as the set of $\alpha=(\alpha_1,\alpha_2,\alpha_3)$ so that $\alpha_i: \overline{\triangle_i} \rightarrow {\bf R}^n  \in W^{1,2}(\triangle_i) \cap C^0(\overline{\triangle_i})$ satisfying 
\[
\alpha_i \circ \phi_i \big|_{A_i} = \alpha_j \circ \phi_j \big|_{A_j} \hspace{0.1in}  \mbox{(matching condition)}
\]
and
\[
\alpha_i \circ \phi_i \big|_{B_i} =\psi_i \hspace{0.1in} \mbox{(boundary  condition)}.
\]
The matching condition says that one can  view $\alpha$ as a continuous map from $X_{\Phi}$ and the boundary condition says that $\alpha$ restricted to $\partial X_{\Phi}$ parameterizes the boundary $\Gamma$ as specified by
$\Psi$ and $\Phi$.  The matching condition implies that $\alpha_i(A_i)=\alpha_j(A_j)$ for $i,j=1,2,3$.  Let $\gamma=\alpha_1(A_1)$.  We will refer to $\gamma$ as the free boundary corresponding to $\alpha$.  
Define the area of $\alpha$ as
\[
A(\alpha)=\sum_{i=1}^3 A(\alpha_i).
\]
We can now formulate the singular version of the classical Plateau Problem as follows: \\
\\
{\bf The singular Plateau Problem.}  {\it Given a uniform parameterization $\Psi$ of  $\Gamma$,  let
\[
{\cal F}= \bigcup_{\Phi \in {\cal P}}{\cal F}_{\Psi}(\Phi).
\] 
Find $\alpha^* \in {\cal F}$ so that $A(\alpha^*) \leq A(\alpha)$ for all $\alpha \in {\cal F}$.  }\\
\\

A well-known method of obtaining the solution of a classical Plateau Problem is by the use of the energy functional.  Motivated by this, we define a relevant notion of energy related to the singular Plateau Problem.  First, we set
\[
{\cal C}=\{c=(c_1,c_2,c_3) \in {\bf R}^3 | c_1+c_2+c_3=1, c_i \geq 0 \mbox{ for } i=1,2,3\}.
\]
For a given  $c \in {\cal C}$, we say that $\alpha \in {\cal F}_{\Psi}(\Phi)$ is compatible with $c$ if $E(\alpha_i)=0$ whenever $c_i=0$.  Otherwise, we say $\alpha$ is incompatible with $c$.  We define the area of $\alpha$ as 
\[
A(\alpha)=\sum_{i=1}^3 A(\alpha_i)
\]
and the $l_2$-energy with weight $c$ as 
\[
E_c(\alpha)=
\left\{
\begin{array}{ll}
\left( \displaystyle{\sum_c} \frac{1}{c_i} (E(\alpha_i))^2  \right)^{1/2} & \mbox{ if $\alpha$ is compatible with $c$}\\
\infty &  \mbox{ if $\alpha$ is in compatible with $c$}
\end{array}
\right.
\]
where $\displaystyle{\sum_c}$ denotes the sum over $i$ with $c_i \neq 0$.   \\

\noindent {\bf The Variational Problem.} \ Find $\Phi^* \in {\cal P}, c^* \in {\cal C}$ and $\alpha^* \in {\cal F}_{\Psi}(\Phi)$ satisfying the equality
\begin{equation} \label{variationalproblem}
E_{c^*}(\alpha^*)= \inf_{\Phi \in {\cal P}} \inf_{ c \in {\cal C}} \inf_{\alpha \in {\cal F}_{\Psi}(\Phi)} E_c(\alpha)
\end{equation}

The following theorem says that solving the variational problem above solves singular Plateau Problem.  \\

\noindent {\bf Theorem} \cite{mese-yamada} 
{\it 
Suppose  $\Phi^* \in {\cal P}, c^* \in {\cal C}$ and $\alpha^* \in {\cal F}_{\Psi}(\Phi^*)$ satisfies (\ref{variationalproblem}).  Then $\alpha$ solves the singular Plateau Problem, i.e.
\[
A(\alpha^*) \leq A(\alpha) \ \mbox{ for all }  \ \alpha \in \bigcup_{\Phi \in {\cal P}} {\cal F}_{\Psi}(\Phi).
\]
Moreover, $\alpha^*_i$ is a weakly conformal map and 
\[
c^*_i= \frac{E(\alpha^*_i)}{\displaystyle{\sum_{c^*}} E(\alpha^*_j) }.
\]
}\\

We have the following existence theorem.  \\

\noindent {\bf Theorem} (cf.~\cite{mese-yamada})
{\it Let ${\cal P}(K)$ be a subset of ${\cal P}$ which are $K$-quasisymmetric.  For each $K \in [1,\infty)$, there exists $\Phi^K \in {\cal P}(K), c^K \in {\cal C}$ and $\alpha^K \in {\cal F}_{\Psi}(\Phi^K)$ satisfying
\begin{equation} \label{Kqs}
E_{c^K}(\alpha^K)= \inf_{\Phi \in {\cal P}(K)} \inf_{ c \in {\cal C}} \inf_{\alpha \in {\cal F}_{\Psi}(\Phi)} E_c(\alpha).
\end{equation} 
 In particular, if
\[
\inf_{\Phi \in {\cal P}} \inf_{ c \in {\cal C}} \inf_{\alpha \in {\cal F}_{\Psi}(\Phi)} E_c(\alpha)=\inf_{\Phi \in {\cal P}(K)} \inf_{ c \in {\cal C}} \inf_{\alpha \in {\cal F}_{\Psi}(\Phi)} E_c(\alpha),
\]
for some sufficiently large $K$, then we can solve the singular Plateau Problem.  }\\

\begin{proof}
Let $\Phi^{(k)}  \in {\cal P}(K), c^{(k)} \in {\cal C}$ be a minimizing sequence in the sense that
\[
\lim_{k \rightarrow \infty} \inf_{\alpha \in {\cal F}(\Phi^{(k)})}E_{c^{(k)}} (\alpha)=\inf_{\Phi \in {\cal P}(K)} \inf_{ c \in {\cal C}} \inf_{\alpha \in {\cal F}_{\Psi}(\Phi)} E_c(\alpha).
\]
We recall that $K$-quasisymmetric maps have $D$-quasiconformal extensions for some $D$ (cf. \cite{leite}).  Let $\Phi^{(k)'}$ be a $D$-quasiconformal extention of $\Phi^{(k)}$.  By a slight modification of the proof of Theorem 14 of \cite{mese-yamada}, there exists a $c^{(k)}$-energy minimizing map $\alpha^{(k)} \in {\cal F}_{\Psi}(\Phi^{(k)})$, i.e.  
\[
E_{c^{(k)}} (\alpha^{(k)}) = \inf_{\alpha \in {\cal F}_{\Psi}(\Phi^{(k)} )} E_{c^{(k)}} (\alpha).
\]  
(We note that the modification of the proof of Theorem 14 only involves the observation that $\alpha_i^{(k)} \circ \varphi_i^{(k)}$ has bounded energy independent of $k$ since $\varphi_i^{(k)}$ is $D$-quasiconformal.) Following the proof of Theorem 16 of \cite{mese-yamada}, we see that (by taking a subsequence of $k$ if necessary) $\Phi^{(k)} \rightarrow \Phi^K, c^{(k)} \rightarrow c^K, \alpha^{(k)}  \rightarrow \alpha^K \in {\cal F}_{\Psi}(\Phi^K)$ and these limits satisfy (\ref{Kqs}).  \end{proof}

Our goal in the following sections is to investigate the minimizing map $\alpha^*$ of (\ref{variationalproblem})  in the case we have an a priori $C^{2,\alpha}$ regularity along with a non-singular condition of the gluing map $\Phi^*$.  

\section{Regularity of energy minimizing maps}

Fix $\Phi \in {\cal P}$ and $c \in {\cal C}$.  We say $\alpha \in {\cal F}_{\Psi}(\Phi)$ is a $c$-weighted $l^2$-energy minimizing map if 
\[
E_c(\alpha) = \inf_{\alpha' \in {\cal F}_{\Psi}(\Phi)} E_c(\alpha').
\]
In this section, we prove higher regularity of  $\alpha$ assuming a regularity condition on the gluing map $\Phi$.  In order to do so, we define another notion of energy which we call the  $l^1$-energy.  

Let 
$k=(k_1,k_2, k_3)$ with $k_i>0$ for  $i=1,2,3$.     Given $\alpha \in {\cal F}_{\Psi}(\Phi)$, the $k$-weighted $l^1$-energy of a map $\alpha$ is given as
\[
E^k(\alpha) = \sum_{i=1}^n k_i E(\alpha_i).
\]
We say $\alpha \in {\cal F}_{\Psi}(\Phi)$ is a $k$-weighted $l^1$-energy minimizing map if

\[
E^k(\alpha) = \inf_{\alpha' \in {\cal F}_{\Psi}(\Phi)} E^k(\alpha')
\]
We prove the following regularity result for an $l^1$-energy minimizing map.  

\begin{theorem} \label{kreg}
For any  $k=(k_1,k_2,k_3)$ with $k_i >0$ and $\Phi \in {\cal P}$, let $\alpha \in {\cal F}_{\Psi}(\Phi)$ be a $k$-weighted $l^1$-energy minimizing map.  If $\Phi \in C^{2,\mu}$, then there exists an extension $\Phi'$  of $\Phi$ so that $\Phi' \in C^{2,\nu}$ for any $0<\nu<\mu$ and $\alpha_i \circ \phi'_i$ is real analytic in $\triangle_i \cup A_i$.  In particular, $\alpha_i \circ \phi_i$ is a real analytic  on $A_i$. 
\end{theorem}

\begin{remark}
Note here that such a gluing data $\Phi$ in $\cal P$ belongs to ${\cal P}(K)$ for some $K < \infty$. Hence the hypothesis
of this theorem fits in the context of the existence theorem~\cite{mese-yamada} referred above.
\end{remark}

\begin{proof}
By \cite{li-tam} (Theorem 4.3), there exists an extension  $\phi'_i : \triangle \rightarrow \triangle$ of $\phi_i : S^1 \rightarrow S^1$ so that $\phi'_i$ is  a harmonic map from the Poincare disk $(\triangle, g_{{\rm hyp}})$ to the Poincare disk, which is quasiconformal and $\phi'_i \in C^{2,\nu}(\overline{\triangle_i})$ for any $0<\nu<\mu$ .   The harmonicity implies that $\phi'_i$ is real analytic in the interior $\triangle_i$.  

Let $\triangle_i^+=\{(x,y) \in \triangle_i: y >0\}$ and $I_i=\{ (x,y) \in \partial \triangle_i^+: y=0\}$.  Let $\rho:\overline{\triangle^+_i} \rightarrow \overline{\triangle_i}$  be the conformal map which fixes $(1,0)$, $(-1,0)$ and $(0,1)$ and set $\phi''_i=\phi'_i \circ \rho$.  For simplicity of exposition, we will prove that $u_i=\alpha_i \circ \phi''_i$ is real analytic on $\triangle_i^+ \cup I_i$.  This proves $\alpha_i \circ \phi'_i=u_i \circ \rho^{-1}$ is real analytic since $\rho^{-1}$ is real analytic on $\triangle_i \cup A_i$.

 Let $^ig$ be the pull-back metric via $\phi''_i:\overline{\triangle_i^+} \rightarrow \overline{\triangle_i}$ of the standard Euclidean metric $g_0$ defined on $\overline{\triangle_i}$.  If we write $^ig=(g_{\alpha \beta})$ in terms of the Euclidean coordinates on $\triangle^+_i \subset {\bf R}^2$( that is $x^1 := x, x^2 := y$,)  then $^ig^{\alpha \beta}$ is real analytic in $\triangle_i^+$ and $C^{1,\nu}$ in $\triangle_i^+ \cup I$.  
Let 
\[
{\cal H}:=\{h=(h_1,h_2,h_3)| h_i:\triangle^+_i \rightarrow {\bf R}^n \in W^{1,2}\},
\]
and consider the functional ${\cal E}:{\cal H} \rightarrow R$ defined by setting
\[
{\cal E}(h)= \sum_{i=1}^3 k_i \int_{\triangle_i^+} \  ^ig^{\alpha \beta} \frac{\partial h_i}{\partial x^{\alpha}} \cdot \frac{\partial h_i}{\partial x^{\beta}} dx^1dx^2.
\]
Set $u=(u_1,u_2,u_3) \in {\cal H}$.  Note that  $\phi''_i$ is an isometry of $(\triangle_i^+, \ ^ig)$ and $(\triangle_i, g_0)$,  and thus
\[
{\cal E}(u)=E^k(\alpha).
\]
This in turn implies that $u$ is the minimizer for ${\cal E}(\cdot)$ with respect to the elements $h \in {\cal H}$  with the same boundary condition as $u$.  By result of \cite{mese} and \cite{daskal-mese1}, $u_i$ is Lipschitz continuous in $\triangle_i^+ \cup I_i$.   

We now claim  $C^{1,\beta}$ regularity of $u_i$ in $\triangle_i^+ \cup I_i$ for $\beta \leq \frac{1}{2}$.  For simplicity, we will use the following conventions.  Let $f=(f_1,f_2,f_3)$ be a set of functions $f_i: \triangle_i^+ \cup I_i \rightarrow {\bf R}^n$.  We set  $\frac{\partial f}{\partial x}=\left( \frac{\partial f_1}{\partial x},\frac{\partial f_2}{\partial x},\frac{\partial f_3}{\partial x}\right)$ and $\frac{\partial f}{\partial y}=\left( \frac{\partial f_1}{\partial y},\frac{\partial f_2}{\partial y},\frac{\partial f_3}{\partial y}\right)$ if $\frac{\partial f_i}{\partial x}, \frac{\partial f_i}{\partial x}$ exist for $i=1,2,3$.   We say  $f$ satisfies the matching condition if
\begin{equation} \label{mf}
f_i(x,0)=f_j(x,0) \mbox{ for } \ i,j=1,2,3
\end{equation}
and balancing condition if
\begin{equation} \label{bf}
\sum_{i=1}^3 \frac{\partial f_i}{\partial y} =0.
\end{equation}
Now assume $f$ satisfies the matching condition.  We write $f \in C^{\infty}$ if $f_i \in C^{\infty}(\triangle_i^+ \cup I_i)$ for $i=1,2,3$.  We write $f \in L^2$ ($L^2_{loc}$, $W^{1,2}$, $W^{1,2}_{loc}$ resp.)  if $f_i \in L^2(\triangle_i^+)$ ($L^2_{loc}(\triangle_i^+)$, $W^{1,2}_{loc}(\triangle_i^+)$, $W^{1,2}(\triangle_i^+)$ resp.)  for $i=1,2,3$.  Furthermore, we set $||f||_{L^2}=\sum_{i=1}^3 ||f_i||_{L^2}$ with other Sobolev norms are defined analogously.  We say $f$ has compact support if $f_i$ has compact support in $\triangle_i^+ \cup I_i$.  We will also set
 \[
Df=\Big( \max \left\{\left| \frac{\partial f_1}{\partial x} \right|, \left| \frac{\partial f_1}{\partial  y} \right| \right\}, \max \left\{\left| \frac{\partial f_2}{\partial x} \right|, \left| \frac{\partial f_2}{\partial  y} \right| \right\}, \max \left\{\left| \frac{\partial f_3}{\partial x} \right|, \left| \frac{\partial f_3}{\partial  y} \right|  \right\} \Big)
\]
and
\[
\delta_h f=\frac{f(x+h,y)-f(x,y)}{h}.
\]
\\
\\
{\bf Claim.}  Define a functional ${\cal O}$ by setting
\begin{equation} \label{start}
{\cal O}(\eta)= \sum_{i=1}^3 \int_{\triangle_i^+} k_i \ ^ig^{\alpha \beta} \frac{\partial f_i}{\partial x^{\alpha}} \frac{\partial \eta_i}{\partial x^{\beta}} dxdy
\end{equation}
for any $\eta \in W^{1,2}$ with compact support.  If $\zeta \in C^{\infty}$ with compact support, $0 \leq \zeta_i\leq 1$, then 
\begin{equation} \label{hh}
||\zeta D (\delta_h f) ||^2_{L^2}  \leq C \left( ||f||_{W^{1,2}}+{\cal O}(\delta_{-h}( \zeta^2 \delta_h f)) \right)
\end{equation}
where  $C$ is a constant depending on $||^ig^{\alpha \beta} ||_{C^{1,\alpha}}$, $k$ and $\zeta$.    \\
\\

\begin{proofclaim}
In this proof, we use $C$ to be a generic constant only dependent on $||^ig^{\alpha \beta}||_{C^{1,\alpha}}$, $k$ and $\zeta$.  We first replace $\eta_i$ by
$
\delta_{-h} \eta_i
$ 
in (\ref{start}).  Since
\[
\delta_h  \left(  \ ^ig^{\alpha \beta} \frac{\partial f_i}{\partial x^{\alpha}} \right) (x,y) =  \ ^ig^{\alpha \beta} (x+h,y) \left( \left( \delta_h \frac{\partial f_i}{\partial x^{\alpha}} \right)  (x,y) \right) + \left( \left( \delta_h \  ^ig^{\alpha \beta} \right) (x) \right) \frac{\partial f_i}{\partial x^{\alpha}}(x,y),
\]
we obtain
\begin{equation} \label{hs}
\sum_{i=1}^3 \int k_i  \ ^ig^{\alpha \beta}(x+h,y) \left( \delta_h \frac{\partial f_i}{\partial x^{\alpha}} \right) \frac{\partial \eta_i}{\partial x^{\beta}}   = -\sum_{i=1}^3 \int k_i(\delta_h \  ^ig^{\alpha \beta}) \frac{\partial f_i}{\partial x^{\alpha}} \frac{\partial \eta_i}{\partial x^{\beta}} +{\cal O}( \delta_{-h} \eta).
\end{equation}
Now replace $\eta_i$ with $\zeta^2 \delta_h f_i$.  If we let RHS and LHS be the right hand side and left hand side of (\ref{hs}), we get
\begin{eqnarray}
\mbox{RHS} & = & -\sum_{i=1}^3 \int k_i(\delta_h \  ^ig^{\alpha \beta}) \frac{\partial f_i}{\partial x^{\alpha}} \left( 2 \zeta \frac{\partial \zeta}{\partial x^{\beta}} \delta_h f + \zeta^2 \frac{\partial}{\partial x^{\beta} }(\delta_h f)  \right) + {\cal O}(\delta_{-h}( \zeta^2 \delta_h f)) 
 \nonumber \\
& \leq & C ||f||_{W^{1,2}} \left( || \delta_h f D\zeta||_{L^2} + ||\zeta (D \delta_h f)||_{L^2} \right) + {\cal O}(\delta_{-h}( \zeta^2 \delta_h f)) 
 \label{rhs}
\end{eqnarray}
  and 
\[
\mbox{LHS} = \sum_{i=1}^3 \int k_i \  ^ig^{\alpha \beta}(x+h,y) \left(\delta_h \frac{\partial f_i}{\partial x^{\alpha}} \right) \left( \delta_h \frac{\partial f_i}{\partial x^{\beta}} \right) \zeta^2 + 2 \int k_i \  ^ig^{\alpha \beta}(x+y,h) \left( \delta_h \frac{\partial f_i}{\partial x^{\alpha}} \right)  \zeta \frac{\partial \zeta}{\partial x^{\beta}} \delta_h f.
\]
The last equality implies that
\begin{equation} \label{lhs'}
\sum_{i=1}^3 \int k_i \  ^ig^{\alpha \beta}(x+h,y) \left(\delta_h \frac{\partial f_i}{\partial x^{\alpha}} \right) \left( \delta_h \frac{\partial f_i}{\partial x^{\beta}} \right) \zeta^2 \leq \mbox{LHS} + C||(D\delta_h f)\zeta||_{L^2} ||D\zeta (\delta_h f)||_{L^2}.
\end{equation}
Therefore, (\ref{hs}), (\ref{rhs}) and (\ref{lhs'}) combine to give
\begin{eqnarray}
\lefteqn{||\zeta (D\delta_h f) ||_{L^2}^2} \nonumber \\
 & \leq & C (\mbox{RHS} +  C ||(D\delta_h f)\zeta||_{L^2} ||D\zeta (\delta_h f)||_{L^2} ) \nonumber \\
& \leq & C \Big( \left. ||f||_{W^{1,2}} ( || \delta_h f D\zeta||_{L^2} + ||\zeta (D \delta_h f)||_{L^2} )  \right. \nonumber \\
& & \left. + {\cal O}(\delta_{-h}( \zeta^2 \delta_h f)) + C||(D\delta_h f)\zeta||_{L^2} ||D\zeta (\delta_h f)||_{L^2}  \Big)\right. \label{gg}.
 \end{eqnarray}
Applying the inequality
$
||\delta_h f||_{L^2{(\Omega')}} \leq ||Df||_{L^2}
$
for any $\Omega' \subset \subset X_{{\rm Id}}$  (cf. Lemma 7.23 \cite{gilbarg-trudinger}) and the Cauchy-Schwarz inequality $2ab \leq \epsilon a^2+\frac{1}{\epsilon} b^2$ several times in  (\ref{gg}),
we obtain
\[
||\zeta D(\delta_h f)||^2_{L^2} \leq C (||f||_{W^{1,2}}^2+  {\cal O}(\delta_{-h}( \zeta^2 \delta_h f)) ).
\]
\end{proofclaim}

Denote the component functions of $u_i$ as $u_i^m$.  Since the map $u$ is ${\cal E}$-minimizing,  each coordinate function $u_i^m$ is also minimizing and the first variation formula gives 

\begin{equation} \label{fv}
0= \sum_{i=1}^3 k_i \int_{\triangle_i}  \  ^ig^{\alpha \beta} \frac{\partial u_i^m}{\partial x^{\alpha}}  \frac{\partial \eta_i}{\partial x^{\beta}} dx^1dx^2
\end{equation}
for any $\eta\in W^{1,2}$ with compact support.  
For any $\Omega_0 \subset \subset X_{{\rm Id}}$, we can then apply
 Claim with $\zeta \equiv 1$ in $\Omega_0$, $f_i=u_i^m$ and ${\cal O}$ equal to 0 to  conclude that
 $\frac{\partial^2 u^m}{(\partial x)^2}, \frac{\partial^2 u^m}{\partial x \partial y} \in L^2(\Omega_0)$, i.e. $\frac{\partial u^m}{\partial x} \in W^{1,2}_{loc}$.  
 
 Now we choose $\eta \in W^{2,2}$ map with compact support with compact support in $\Omega \subset \subset X_{{\rm Id}}$.  Since $\frac{\partial u^m}{\partial x} \in W^{1,2}(\Omega)$, we can integrate the pointwise equality
 \[
 \frac{\partial}{\partial x} \left( \ ^ig^{\alpha \beta} \frac{\partial u_i^m}{\partial x^{\alpha}} \frac{\partial \eta_i}{\partial x^{\beta}}\right)   
=  \left(  \frac{\partial}{\partial x} \ ^ig^{\alpha \beta} \right)\frac{\partial u_i^m}{\partial x^{\alpha}} \frac{\partial \eta_i}{\partial x^{\beta}}
+\ ^ig^{\alpha \beta} \frac{\partial }{\partial x^{\alpha}} \left( \frac{\partial u_i^m}{\partial x} \right) \frac{\partial \eta_i}{\partial x^{\beta}}
+ \ ^ig^{\alpha \beta} \frac{\partial u_i^m}{\partial x^{\alpha}} \frac{\partial}{\partial x^{\beta}} \left(\frac{\partial \eta_i}{\partial x}\right).
\]
Note that   the left hand side of the resulting integral equality is equal to 0 since $x \mapsto \eta_i(x,y)$ has compact support for all $y$.  Furthermore, the third term on the right hand side of the resulting integral equality is equals 0 by (\ref{fv}) since $\frac{\partial \eta}{\partial x} \in W^{1,2}$ with compact support.  Thus, we obtain
 \[
-  \sum_{i=1}^3 \int k_i  \frac{\partial}{\partial x} \ ^ig^{\alpha \beta} \frac{\partial u_i^m}{\partial x^{\alpha}} \frac{\partial \eta_i}{\partial x^{\beta}} =
\sum_{i=1}^3 \int k_i \ ^ig^{\alpha \beta} \frac{\partial }{\partial x^{\alpha}} \left( \frac{\partial u_i^m}{\partial x} \right) \frac{\partial \eta_i}{\partial x^{\beta}}.
\]
Let $\Omega' \subset \subset \Omega$ and let $\zeta \in C^{\infty}$ with $0 \leq \zeta_i \leq 1$, $\zeta \equiv 1$ in $\Omega'$ and compact support in $\Omega$.  We  apply the claim with $f_i=\frac{\partial u_i^m}{\partial x}$ and
 ${\cal O}(\eta)$ equal to the left hand side of the equality above to obtain
\[
||\zeta D(\delta_h \frac{\partial u^m}{\partial x})||^2_{L^2} \leq C \left((||\frac{\partial u^m}{\partial x}||^2_{W^{1,2}(\Omega)}+ {\cal O}(\eta) \right).
\]
Substituting  $\eta_i$ with $\delta_{-h}(\zeta^2 \delta_h \frac{\partial u_i^m}{\partial x})$ in ${\cal O}(\eta)$, we obtain
\begin{eqnarray*}
{\cal O}(\delta_{-h}(\zeta^2 \delta_h \frac{\partial u_i^m}{\partial x})) & = &  -  \sum_{i=1}^3 \int k_i  \frac{\partial}{\partial x} \ ^ig^{\alpha \beta} \frac{\partial u_i^m}{\partial x^{\alpha}} \left( \delta_{-h} \zeta^2 \delta_h  \frac{\partial u_i^m}{\partial x^{\beta}} \right) \\
&  \leq & C||u^m||_{W^{1,2}} \left(||\zeta \delta_{-h} \delta_h \frac{\partial u^m}{\partial x^{\beta}}||_{L^2} + ||D \frac{\partial u^m}{\partial x}||_{L^2(\Omega)} \right) \\
& \leq & C||u^m||_{W^{1,2}} \left(||\zeta D (\delta_h \frac{\partial u^m}{\partial x})||_{L^2} + ||D \frac{\partial u^m}{\partial x}||_{L^2(\Omega)} \right) \\
& \leq &  C\left( \frac{1}{\epsilon} ||u^m||^2_{W^{1,2}} + \frac{\epsilon}{2} ||\zeta D (\delta_h \frac{\partial u^m}{\partial x})||^2_{L^2} + \frac{\epsilon}{2} ||D \frac{\partial u^m}{\partial x}||^2_{L^2(\Omega)} \right) .
\end{eqnarray*}
The above two inequalities combine to show 
\[
||\zeta D(\delta_h \frac{\partial u^m}{\partial x})||^2_{L^2} \leq C ||\frac{\partial u^m}{\partial x}||^2_{W^{1,2}(\Omega)}
\]
which in turn implies
\[
||D(\frac{\partial^2 u^m}{(\partial x)^2})||^2_{L^2(\Omega')} \leq C ||\frac{\partial u^m}{\partial x}||^2_{W^{1,2}(\Omega)}.
\]
By the $W^{1,2}$-trace theory,  $\frac{\partial^2 u_i^m}{(\partial x)^2} \in W^{1,2}_{loc}(\triangle_i^+ \cup I_i)$  implies that $\frac{\partial^2 u_i^m}{(\partial x)^2} \in L^2_{loc}(I_i)$.  
  Therefore, for $(a,0), (b,0) \in I_i$, 
\[
\left| \frac{\partial u^m_i}{\partial x}(b,0)-\frac{\partial
u^m_i}{\partial x}(a,0) \right| = \int_a^b \frac{\partial^2 u^m_i}{\partial
x^2} (x,0) dx \leq  (b-a)^{1/p}  \left( \int_{I_i} \left| \frac{\partial^2
u^m_i}{\partial x^2}\right|^{q} \right)^{1/q},
\]
for any $q \leq
2$ (and hence any $p \geq 2$.)  This implies that $u^m_i$ restricted
to $A$ is $C^{1,\beta}$ for any $\beta \leq 1/2$. Thus, by the
boundary regularity of harmonic maps, we have that $u^m_i\in C^{1,\beta}(\triangle_i^+ \cup I_i)$.

We will now assume, for the moment, that $u^m_i \in C^2(\triangle_i^+ \cup I_i)$  and prove that $u^m_i$ is real analytic in $\triangle_i^+ \cup I_i$.  The minimizing property of $u_i^m$ says that $u_i^m$ is a smooth harmonic function in $\triangle_i^+$ (with respect to the metric $^ig$) and satisfies the harmonic map equation
\[
0=\frac{\partial }{\partial x^{\beta} } \left( \ ^ig^{\alpha \beta} \frac{\partial u_i^m}{\partial x^{\alpha}}  \right).
\]
Multiplying the above equality by $\zeta =(\zeta_1,\zeta_2,\zeta_3) \in C^{\infty}$ with compact support and integrating, we obtain
\[
0= \sum_{i=1}^3  \int_{\triangle_i^+} k_i \frac{\partial }{\partial x^{\beta} } \left( \ ^ig^{\alpha \beta} \frac{\partial u_i^m}{\partial x^{\alpha}}  \right) \zeta_i
 = \sum_{i=1}^3 k_i \int_{A_i} \  ^ig^{\alpha 2} \frac{\partial u_i^m}{\partial x^{\alpha}} \zeta_i dx.
\]
Since $u^m_i \in C^{1,\alpha}(A_i)$, we obtain  the following linear system satisfied by $3n$ unknown variables $u_i^m$:
\begin{eqnarray*}
(E_i^m)  &  \displaystyle{  \frac{\partial}{\partial x^{\alpha}}
\left( \ ^ig^{\alpha \beta} \frac{\partial u_i^m}{\partial
x^{\beta}}\right) =0 } &   \mbox{ for }  i=1, 2, 3, \,\,\,\, m=1,...,n \\
(B_i^m) &  u_i^m-u_1^m=0 &   \mbox{ for } i=2,3, \,\,\,\, m=1,...,n \\
  (C_1^m) &   \displaystyle{ \sum_{i=1}^3 k_i  \, {}^ig^{\alpha 2} \frac{\partial u_i^m}{\partial x^{\alpha}} =0} & \mbox{ for } m=1,...,n.
  \end{eqnarray*}
  We wish to assign a weight the equations, unknowns and the boundary condition so that
  the system given above is elliptic and coercive as in [KNS].  We choose a weight
  $s_i^m=0$ to each equation $(E_i^m)$, $t_i^m=2$ to each unknown variable $u_i^m$,
  $r_i^m=-2$ $(i=2, 3)$ to each boundary condition $(B_i^m)$ and $r_1^m=-1$ to
  each boundary condition $(C_1^m)$.  The principle part of $(E_i^m)$,
  (resp. $(B_i^m)$, $(C_1^m))$ is the part of exactly order $2=s_i^m+t_i^m$
  (resp. $0=r_i^m+t_i^m$ for $i=2, 3$, $1=r_i^m+t_1^m$).
Thus, we get the principle part of the linear system corresponding to these choices of weights as:
 \begin{eqnarray*}
(\tilde{E}_i^m) &   \displaystyle{ {}^ig^{\alpha \beta} \frac{\partial^2 u_i^m}{\partial x^{\alpha} \partial x^{\beta}} }
 =0 & \mbox{ for } i=1, 2, 3, \,\,\,\, m=1,...,n \\
(\tilde{B}_i^m) &    u_i^m-u_1^m=0 &  \mbox{ for } i=2,3, \,\,\,\, m=1,...,n \\
  (\tilde{C}_1^m) &  \displaystyle{ \sum_{i=1}^3 k_i \, {}^ig^{2 \beta} \frac{\partial u_i^m}{\partial x^{\beta}} =0,} &  \mbox{ for } m=1,...,n.
  \end{eqnarray*}

To show our system is elliptic, it is enough to show that any solution of
$(\tilde{E}_i^m)$ of the form
\[
u_i^m=c_i^m e^{\sqrt{-1} (x\xi +y \eta)}, \ \mbox{ with } 0 \neq (\xi,\eta) \in {\bf R}^2
\]
is trivial.  From
\[
^ig^{\alpha \beta}\frac{\partial^2 u_i^m}{\partial x^{\alpha} \partial x^{\beta}}=(^ig^{11} c_i^m \xi^2 + 2\ ^ig^{12} \sqrt{-1} c_i^m \xi \eta + \ ^ig^{22} c_i^m \eta^2) e^{\sqrt{-1}(x\xi+y\eta)} =0
\]
we obtain
\[
\left\{
\begin{array}{l}
^ig^{11} c_i^m \xi^2 +\ ^ig^{22} c_i^m \eta^2 =0\\
2\  ^ig^{12} c_i^m \xi \eta =0.
\end{array}
\right.
\]
Since $(\xi,\eta) \neq 0$, this implies $c_i^m=0$.  Hence $(\tilde{E}_i^m)$ is elliptic.

To show our system is coercive, we need to show that any bounded solution $(\tilde{E}_i^m)$ of the form
\[
u_i^m=e^{\sqrt{-1} \xi x} \varphi_i^m(y), \mbox{ with } 0 \neq \xi \in {\bf R}
\]
satisfying $(\tilde{B}_i^m)$ is trivial. From $(\tilde{E}_i^m)$ we
obtain,
\[
^ig^{\alpha \beta}\frac{\partial^2 u_i^m}{\partial x^{\alpha} \partial x^{\beta}}=(- \ ^ig^{11}
\xi^2 \varphi_i^m + 2 \ ^ig^{12}\sqrt{-1}  \xi \frac{d \varphi_i^m}{dt}
+ \ ^ig^{22} \frac{d^2 \varphi_i^m}{dt^2} ) e^{\sqrt{-1} x\xi} =0
\]
and hence $\varphi_i^m=c_i^m e^{(a_i^m+\sqrt{-1}b_i^m)y}$ where
$z=a_i^m+\sqrt{-1}b_i^m$ is a root of the characteristic equation
\[
^ig^{22}z^2 + 2 \ ^ig^{12}\sqrt{-1}\xi z- \ ^ig^{11}\xi^2 =0.
\]
Since $u_i^m$ is bounded, $a_i^m\leq 0$ for all $i=1, 2, 3$ and $m=1,...,n$.  The boundary condition $(\tilde{B}_i^m)$ implies
\[
c_i^m e^{\sqrt{-1}\xi x} =c_1^m e^{\sqrt{-1}\xi x}
\]
and hence $C:=c_1=...=c_i$.  The other boundary condition
$(\tilde{B}_1^m)$ implies
\[
C \sum_{k=1}^m a_i^m e^{\sqrt{-1}\xi x}=0.
\]
Thus, the condition that $a_i^m\leq 0$ implies that $C=0$.  This
shows $u_i^m \equiv 0$.

Now that we have shown our system to be elliptic and coercive with
the chosen weights $s_i^m=0$, $t_i^m=0$ and $r_i^m=-2$ ($k\neq 1$)
and $r_1^m=-1$, we can apply the elliptic regularity theorem of \cite{KNS} (see also \cite{ADN2} and \cite{Mo}).  This implies that $u_i^m \in
C^{\omega}(\triangle_i^+ \cup I_i)$. 

 We are now left with proving that $u_i^m \in C^{1,\beta}$ implies $u_i^m \in C^2$.  Indeed,  the assertion follows from writing down the system $(E^m_i)$ in divergence form then applying the Schauder estimate for a system in \cite{ADN2} combined with a standard difference quotient argument (cf. \cite{KNS} Lemma 5.1 and the remark following it).
 \end{proof}

In order to prove higher regularity of an $l^2$-energy minimizer, it is now sufficient to show that an $l^2$-minimizers is an $l^1$-minimizer.  We first need the following convexity and uniqueness statement for both the $l^1$- and $l^2$-energy
minimizers.  

\begin{lemma} \label{uniqueness}
Given two $W^{1,2}$ maps $\alpha^{(0)}, \alpha^{(1)}:X_{\Phi}
\rightarrow {\bf R}^n$ in ${\cal F}_{\Psi}(\Phi)$, define
\[
\alpha^{(t)}=(\alpha^{(t)}_1,\dots,\alpha^{(t)}_n):X_{\Phi} \rightarrow {\bf R}^n
 \] by setting
$\alpha^{(t)}_i:=(1-t) \alpha^{(0)}_i + t\alpha^{(1)}_i$. Then
\begin{equation} \label{convexitystatement}
 E^k(\alpha^{(t)}) \leq (1-t)
E^k(\alpha^{(0)})+tE^k(\alpha^{(1)})-t(1-t)\sum_{i=1}^n k_i \int
|\alpha^{(0)}_i-\alpha^{(1)}_i|^2 dxdy.
\end{equation}
If a continuous $W^{1,2}$ map $\alpha:X_{\Phi} \rightarrow {\bf
R}^n$ is $k$-weighted $l^1$-energy minimizing in ${\cal F}_{\Psi}(\Phi)$, then $\alpha$ is the unique $k$-weighted $l^1$-energy
minimizing map.  Similarly, if $\alpha$ is a $c$-weighted $l^2$-energy minimizing in ${\cal F}_{\Psi}(\Phi)$, then $\alpha$ is the unique $c$-weighted $l^2$-energy minimizing map.
\end{lemma}

\begin{proof}
  The
convexity  of the Dirichlet energy  (see for example,
\cite{hartman} or \cite{SY} Chapter X (2.6ii)),
\begin{equation} \label{cvxty}
 E(\alpha^{(t)}_i) \leq (1-t)
E(\alpha^{(0)}_i)+tE(\alpha^{(1)}_i)-t(1-t)\int
|\alpha^{(0)}_i-\alpha^{(1)}_i|^2 dxdy
\end{equation}
 immediately implies
inequality~(\ref{convexitystatement}). We note here that the $\alpha^{(t)}$'s share the same Dirichlet boundary
condition specified by $\Psi$ and $\Phi$.  Thus, if $\alpha^{(0)}$ and
$\alpha^{(1)}$ are both minimizing, then
\[
\int |\alpha^{(0)}_i-\alpha^{(1)}_i|^2 dxdy=0
\]
for every $i=1, \dots, n$ and  this immediately implies the
uniqueness of $k$-weighted $l^1$-energy minimizer.  Furthermore, if $\alpha_i^{(0)} \neq \alpha_i^{(1)}$, then  Jensen's inequality applied to (\ref{cvxty}) implies
\[
 E(\alpha^{(t)}_i)^2  < (1-t)
E(\alpha^{(0)}_i)^2+tE(\alpha^{(1)}_i)^2
\]
and this implies the uniqueness of the $c$-weighted $l^2$-minimizing map.
\end{proof}

\begin{lemma} \label{k-min/c-min}
If ${\tilde{\alpha}}$ is the $k$-weighted $l^1$-energy minimizing map in ${\cal F}_{\Psi}(\Phi)$,
then it is the $c$-weighted $l^2$-energy minimizing map in ${\cal F}_{\Psi}(\Phi)$ with
\[
c=\left( \frac{E(\tilde{\alpha}_1)}{k_1 \sum_{j=1}^n
(E(\tilde{\alpha}_j)/k_j)}, \dots , \frac{E(\tilde{\alpha}_n)}{k_n
\sum_{j=1}^n (E(\tilde{\alpha}_j)/k_j)} \right).
\]
  If $\tilde{\alpha}$ is the $c$-weighted $l^2$-energy minimizing map in ${\cal F}_{\Psi}(\Phi)$, then $\tilde{\alpha}$ is the $k$-weighted $l^1$-energy minimizing map in ${\cal F}_{\Psi}(\Phi)$ with
\begin{equation} \label{k}
k=\left(\frac{E(\tilde{\alpha}_1)}{x_1}, \dots, \frac{E(\tilde{\alpha_n})}{x_n} \right).
\end{equation}
Here,  $x=(x_1,...,x_n)$ is the solution to the linear system $(C-I)x=0$ where $C$ is a $n \times n$ matrix  with all columns equal to $c$ and $I$ is the identity matrix.
\end{lemma}
\begin{proof}
To see that $\tilde{\alpha}$ is $c$-harmonic, let
$\tilde{\alpha}'$ be a competitor of $\tilde{\alpha}$. We are
assuming that
\begin{equation} \label{minimizing}
\sum_{i=1}^n k_i E(\tilde{\alpha}_i) \leq \sum_{i=1}^n
k_iE(\tilde{\alpha}'_i).
\end{equation}
Use (\ref{minimizing}) to show that
\begin{eqnarray}
E_c(\tilde{\alpha})^2& = & \sum_{i=1}^n \frac{1}{c_i}
E(\tilde{\alpha}_i)^2
 \nonumber \\
& = &  \sum_{i=1}^n \frac{1}{ E(\tilde{\alpha}_i) /(k_i
\sum_{j=1}^n
(E(\tilde{\alpha}_j)/k_j)) } E(\tilde{\alpha}_i)^2  \nonumber \\
& = & \sum_{i=1}^n k_i E(\tilde{\alpha}_i) \sum_{j=1}^n
\frac{1}{k_j} E(\tilde{\alpha}_j)
 \label{pp} \\
 & \leq & \sum_{i=1}^n k_i E(\tilde{\alpha}'_i) \sum_{j=1}^n
 \frac{1}{k_j}
E(\tilde{\alpha}_j) \nonumber \\
& \leq & \sum_{i=1}^n \frac{1}{1/\sum_{j=1}^n
(E(\tilde{\alpha}_j)/k_j)} k_i
E(\tilde{\alpha}'_i) \nonumber \\
& \leq & \sum_{i=1}^n \frac{1}{E(\tilde{\alpha}_i) /(k_i
\sum_{j=1}^n (E(\tilde{\alpha}_j)/k_j))}
E(\tilde{\alpha}_i) E(\tilde{\alpha}'_i) \nonumber \\
& \leq & \sum_{i=1}^n \frac{1}{c_i} E(\tilde{\alpha})
E(\tilde{\alpha}'_j)
\nonumber \\
& \leq & \left( \sum_{i=1}^n \frac{1}{c_i}E(\tilde{\alpha}_i)^2
\right)^{1/2} \left( \sum_{i=1}^n
\frac{1}{c_i}E(\tilde{\alpha}'_i)^2 \right)^{1/2} \nonumber \\
& = & E_c(\tilde{\alpha}) \cdot E_c(\tilde{\alpha}'). \nonumber
\end{eqnarray}
Thus, we have shown
 $E_c(\tilde{\alpha}) \leq E_c(\tilde{\alpha}')$ and this proves $\tilde{\alpha}$ is also $c$-weighted $l^2$-energy minimizing.

 To prove the last statement of the lemma, let $\tilde{\alpha}$ be the   $c$-weighted $l^2$-energy minimizing, $k$ as in (\ref{k}) and  $\tilde{\alpha}'$ the unique $k$-weighted $l^1$-energy minimizing map.  Since
 \[
 (C-I)x=0 \Leftrightarrow c_i \sum_{i=1}^n \frac{E(\tilde{\alpha}_j)}{k_j} =\frac{E(\tilde{\alpha}_i)}{k_i} \Leftrightarrow c_i = \frac{E(\tilde{\alpha}_i)}{k_i \sum_{j=1}^n (E(\tilde{\alpha}_j)/k_j)} ,
 \]
the previous paragraph implies that  $\tilde{\alpha}'$ is also the  $c$-weighted $l^2$-energy minimizer.  The uniqueness of $c$-weighted $l^2$-energy minimizer implies that $\tilde{\alpha}=\tilde{\alpha}'$ and hence $\tilde{\alpha}$ is also the $k$-energy $l^2$-energy minimizer.
\end{proof}

\begin{theorem} \label{creg}
Let $c \in {\cal C}$ and  $\Phi \in {\cal F}$ be so that $\Phi \in C^{2,\mu}$ and non-singular (i.e $\frac{\partial \phi_i}{\partial \theta} \neq 0)$. If  $\alpha \in {\cal F}_{\Psi}(\Phi)$  is a $l_2$-energy minimizer with weight $c$, then  there exists an extension $\Phi' \in C^{2,\nu}$ of $\Phi$ for any $0 < \nu<\mu$ so that $\alpha_i \circ \phi'_i \in C^{\omega}(\triangle_i \cup A_i)$.     In particular, $\alpha_i \circ \phi_i$ is a real analytic function on $A_i$.  Moreover, $\alpha_i \in C^{1,\gamma}(\triangle_i \cup A_i)$ for any $0<\gamma<1$.  
\end{theorem}

\begin{proof}
The existence of $\Phi' \in C^{2,\nu}$  so that that $\alpha_i \circ \phi_i' \in C^{\omega}(\triangle_i \cup A_i)$ follows immediately from Theorem~\ref{kreg} and Lemma~\ref{k-min/c-min}.   Furthermore, if we let $\varphi_i =\phi_i^{-1}$, then 
\[
\frac{\partial \varphi_i}{\partial \theta}(\theta)=\frac{1}{\frac{\partial \phi_i}{\partial \theta}(\varphi_i(\theta))}.
\]

Since $\frac{\partial \phi_i}{\partial \theta}(\varphi_i(\theta)) \neq 0$ this implies that  $\varphi_i$ is Lipschitz.  Moreover,
\begin{eqnarray*}
\left| \frac{\partial \varphi_i}{\partial \theta}(\theta_1) - \frac{\partial \varphi_i}{\partial \theta}(\theta_2) \right| & = & \left| \frac{1}{\frac{\partial \phi_i}{\partial \theta}(\varphi_i(\theta_1))} - \frac{1}{\frac{\partial \phi_i}{\partial \theta}(\varphi_i(\theta))} \right| \\
& = &  \frac{1}{\frac{\partial \phi_i}{\partial \theta}(\varphi_i(\theta_1)) \frac{\partial \phi_i}{\partial \theta}(\varphi_i(\theta_2))} \ \left| \frac{\partial \phi_i}{\partial \theta}(\varphi_i(\theta_1))-\frac{\partial \phi_i}{\partial \theta}(\varphi_i(\theta_2)) \right|
\end{eqnarray*}
which immediately shows that $\frac{\partial \varphi_i'}{\partial \theta}$ is Lipschitz and hence $\varphi_i \in C^{1,1}$.  Thus, $\alpha_i \big|_{A_i}= \alpha_i \circ \phi_i \circ \varphi_i \in C^{1,1}$.  Boundary regularity theorem for harmonic maps then implies that $\alpha_i \in C^{1,\gamma}(\triangle_i \cup A_i)$ for any $0<\gamma<1$. 
\end{proof}

\section{Isothermal coordinates}

 Let $\triangle_i^+=\{(x,y) \in \triangle_i:  y>0\}$ and $I_i=\{(x,y) \in \partial \triangle_i: y=0\}$ as in the proof of Theorem~\ref{kreg}.  
Let $Y_0$ be the space which is the union of  three closed half disks $\overline{\triangle^+_1}, \overline{\triangle^+_2}, \overline{\triangle^+_3}$ identified along the $x$-axis $I_1, I_2, I_3$ by the identity map.
We denote the $I_i$'s in $Y_0$ by $I$.   
A map $f=(f_1,f_2,f_3):Y_0 \rightarrow {\bf R}^n$ is said to be conformal if $f_i: \overline{\triangle_i^+} \rightarrow {\bf R}^n$  is conformal (with respect to the standard Euclidean metric on $\triangle_i^+$) and satisfies the matching condition $f_i(x,0)=f_j(x,0)$ for all $-1<x<1$ and $i,j=1,2,3$.

Let ${\rm Id}=({\rm id}_1,{\rm id}_2,{\rm id}_3)$ be the element of ${\cal P}$ so that ${\rm id}_i:\partial \triangle_i \rightarrow \partial \triangle_i$ is the identity map.   With this notation, let $X_{{\rm Id}}$ be as in (\ref{space}) with $\Phi={\rm Id}$.  We denote the $A_i$'s in $X_{{\rm Id}}$ by $A$.   A metric on $X_{{\rm Id}}$ is $g=\{^ig\}$ so that $^ig$ is a metric on $\triangle_i$.  Let $F=(F_1,F_2,F_3) :Y_0 \rightarrow (X_{{\rm Id}}, g)$ be a  map which is injective but not necessarily surjective so that $F_i(\overline{\triangle_i^+})  \subset \overline{\triangle_i}$.   We say that $F$ is an isothermal coordinate system at $p  \in A$  if $F_i(0,0)=p$ , $F_i$ is a conformal map from $\triangle^+_i$ to $(\triangle_i,\ ^ig)$ and $F_i(x,0) =F_j(x,0)$ for $i,j=1,2,3$.  

Let $\Phi^* \in {\cal P}, c^* \in {\cal C}$ and $\alpha^* \in {\cal F}(\Phi^*)$  satisfy the equality (\ref{variationalproblem}) with $\Phi^* \in C^{2,\mu}$ and non-singular.   Then $\alpha^*_i$ is a weakly conformal harmonic map which in $C^{1,\gamma}(\triangle_i \cup A_i)$ by  Theorem~\ref{creg}.   Thus, the differential $d\alpha^*_i$ can be continuously extended to be defined on $\triangle_i \cup A_i$.   
Let $S_i \subset \triangle_i \cup A_i$ be the set of isolated points where the $d\alpha^*_i$ is rank zero.  We define the singular set of $\alpha^*$ as
\[
S=\{ p \in X_{\Phi^*}: p \in S_i \mbox{ for some }i\}.
\]
and the regular set $R$ to be  $X_{\Phi^*} \backslash S$.

Let $p$ be a point
in $R \cap A$.  Then near the point $\alpha^*_i(p)$, the image of the map is a collection of three embedded minimal surfaces  meeting along a curve.  This curve is real analytic since it is the image of real analytic map $\alpha^*_i \circ \phi_i \big|_{A_i}$ by Theorem~\ref{creg}.    We can thus apply Corollary~12 of  \cite{mese-yamada}
to see that there exists a neighborhood $U$ of $p$ so that  three minimal surfaces $\alpha^*_i(\triangle_i \cap U)$ meet along
the free boundary $e_{\alpha^*}:=\alpha^*_i(A_i \cap U)$ at
$120^o$ angles.

We summarize the preceding statements as follows.

\begin{theorem} \label{three}
Assume $\Phi^* \in {\cal P}, c^* \in {\cal C}$ and $\alpha^* \in {\cal F}(\Phi^*)$  satisfies the equality (\ref{variationalproblem}).  If $\Phi^* \in C^{2,\alpha}$ and non-singular, then for every regular point $p \in R \cap A \subset X_{{\rm Id}}$, there is a neighborhood of 
the point $\alpha^*(p)$, where the image of $\alpha^*$ is an union of three embedded surfaces $\Sigma_1,\Sigma_2$ and $\Sigma_3$ meeting along an embedded real analytic curve $\gamma$.   The unit normal vectors $\eta_1,\eta_2,\eta_3$ of $\Sigma_1,\Sigma_2,\Sigma_3$ along $\gamma$ satisfies the geometric balancing condition:
\begin{equation} \label{geombal}
\sum_{i=1}^3 \eta_i=0
\end{equation}
at every point of $\gamma$.   
\end{theorem}

By reparameterizing $X_{{\rm Id}}$, we have the following local uniformization of the singular minimal surfaces near $\alpha(p)$ for every regular point $p \in R \cap A$.

\begin{theorem} \label{isothermal}
Assume $\Phi^* \in {\cal P}, c^* \in {\cal C}$ and $\alpha^* \in {\cal F}(\Phi^*)$  satisfies the equality (\ref{variationalproblem}).  If $\Phi^* \in C^{2,\alpha}$ and non-singular, then for every regular point $p \in R \cap A \subset X_{{\rm Id}}$, there exists conformal, harmonic map $f =(f_1,f_2,f_3):Y_0 \rightarrow {\bf R}^n$   parameterizing the image of $\alpha$ in a neighborhood of $\alpha(p)$
\end{theorem}

\begin{proof}
By Theorem~\ref{creg}, there exists an extension $\Phi'$ of $\Phi^*$ so that $\alpha^*_i \circ \phi_i$ is real analytic in $\triangle_i \cup A_i$.    Let $^iG$ be the the pull back of the Euclidean metric on ${\bf R}^n$ under the map $\alpha^*_i \circ \phi'_i$ and $^ig$ be the pull back of the Euclidean metric on $\triangle_i$ under the map $\phi'_i$.  By Theorem~\ref{creg}, $^iG$ is $C^{\omega}(\triangle_i \cup A_i)$ and  $^ig$ is $C^{1,\gamma}(\triangle_i \cup A_i)$.   As a $c^*$-weighted $l^2$-energy minimizer, $\alpha^*_i$ is  weakly conformal and hence $\alpha^*_i \circ \phi'_i$  is weakly conformal with respect to  $^ig$.   Therefore  $^iG$  is conformally equivalent to $^ig$, i.e. $^iG=\lambda ^ig$ for some nonnegative function $\lambda$.    With respect to the Euclidean coordinates of $\triangle_i$, denote the metric components
of $^ig$ by $^ig_{\alpha \beta}$, and $^iG$ by $^iG_{\alpha \beta}$.  We  now wish find an isothermal coordinate by solving the Beltrami equation
\[
w_{\overline{z}} = \mu w_z
\]
where Beltrami coefficient $\mu$ is given by
\[
\mu = \frac{ ^ig_{11}- \ ^ig_{22} +2\sqrt{-1} \ ^ig_{12} }{^ig_{11}+^ig_{22}+2\sqrt{ ^ig_{11} \ ^ig_{22}- \ ^ig_{12}^2}}=\frac{ ^iG_{11}- \ ^iG_{22} +2\sqrt{-1} \ ^iG_{12} }{^iG_{11}+\ ^iG_{22}+2\sqrt{ ^iG_{11} \ ^iG_{22}- \ ^iG_{12}^2}}.
\]
Note here that the Beltrami coefficient $\mu$ is represented by the pull-back metric $^iG=(\alpha_i \circ \phi'_i)^* G_0$ as well as  by the pull-back metric
$^ig=(\phi'_i)^* g_0$.  As $\phi'_i$ is a $C^{2, \alpha}$-diffeomorphism of bounded dilatation, the Beltrami coefficient $\mu$ has moduli strictly less than one 
and is $C^{1, \alpha}$.  Furthermore, the metric components $^iG_{\alpha \beta}$ are given by
\[
^iG_{\alpha \beta} = \langle \frac{\partial (\alpha_i \circ \phi'_i)}{\partial x^{\alpha}}, \frac{\partial (\alpha_i \circ \phi'_i)}{\partial x^{\beta}} \rangle_{{\bf R}^n}.
\] 
Thus, the components of $^iG$ are real analytic on $\triangle \cup A_i$ since the map $\alpha_i  \circ \phi'_i$ is real analytic there.    Note that the quantity $\sqrt{^iG_{11}\ ^iG_{22}- \ ^iG_{12}^2}$ is the 
 pulled-back area form of the immersed surface $\alpha_i \circ \phi'_i (\triangle_i \cup A_i)$ possibly with isolated branch points in ${\bf R}^n$. 
Away from the branch points, namely in a neighborhood where the differential of the map $\alpha_i \circ \phi_i$ is non-degenerate, 
the quantity $\sqrt{^iG_{11} \ ^iG_{22}- \ ^iG_{12}^2}$ is real analytic, as the term $^iG_{11} \ ^iG_{22}-^iG_{12}^2$ is strictly positive.
Hence $\mu$ is a ratio of two real analytic functions 
on the open set $R$ where the differential of the map $\alpha_i \circ \phi_i$ is of rank two
as the latter function is strictly positive on the set.  


Recall here that there is the conformal reparameterization of the unit disk $\triangle_i$ by the half disk $\triangle^+_i$
via the Riemann mapping $\rho: \triangle^+_i \rightarrow \triangle_i$ fixing $(1,0), (-1,0)$ and $(0, 1)$.    Pull back the Beltrami coefficient $\mu_i$ (which is covariant as $\rho$ is conformal) and the regular set $R$ defined on the disk, onto the half disks by the map $\rho$. By abuse of notation, we use $\mu_i$ and $R$ for the pulled-back coefficient and the pulled-back regular set respectively. 

Let $P = \alpha_i \circ \phi'_i (p)$ be a point on the free boundary with $p= \rho (x_0, 0) \in R \cap A$.  We  set $w(z, \overline{z}) = u(x, y) +iv$
and $\mu = \eta (x,y) + i \zeta (x, y)$ to rewrite the Beltrami 
equation defined on the half disk $\triangle^+$ as the following system of equations with real analytic coefficients:
\[ 
\left( \begin{array}{c}
u_y \\
v_y \end{array} \right)
=
\left( \begin{array}{cc}
\zeta & (1+\eta)  \\
(1-\eta) & \zeta  \end{array} \right)^{-1} 
\left( \begin{array}{c}
(1-\eta)u_x + \zeta v_x \\
-\zeta u_x - (1+\eta) v_x \end{array}
\right)
\] 
The inverse matrix on the  right hand side exists because $|\mu|^2 = \eta^2 + \zeta^2 < 1$.  We also have the Cauchy initial data
\[
u(x, 0) = x-x_0 \mbox{ and } v(x, 0) = 0
\]
for $(x,0) \in I_i$ near $(x_0, 0) \in I_i$.
Therefore, we can apply the Cauchy-Kowalewski Theorem and obtain, in some neighborhood of the point $p$, a unique solution to the Beltrami equation.  This solution $w_i$ is a quasiconformal diffeomorphism from a neighborhood ${\cal U}_i \subset \triangle_i^+ \cup I_i$ of $(x_0, 0)$ to a neighborhood ${\cal V}_i \subset \triangle_i^+ \cup I_i$ of $(0,0)$.  
By construction, the  pulled-back metric of the Euclidean metric $g_0$ of $\triangle_i^+$  under  $w_i$ is conformal to $\rho^*(^ig)$.   Thus the map $w_i^{-1}$ provides a parameterization of the neighborhood of 
$\rho^{-1} (p) =(x_0, 0)$ in $(\triangle^+_i \cup I_i, \rho^*({}^ig))$ by a open set ${\cal V}_i$ in $\triangle^+_i \cup I_i$.  After  scaling and then
composing with $\rho$, we have constructed an isothermal coordinate system $F=(F_1,F_2,F_3):Y_0 \rightarrow (X_{{\rm Id}},g)$ of $p \in A$.  The map $f=(f_1,f_2,f_3)$ with $f_i = \alpha_i \circ \phi_i' \circ F$ satisfies the desired properties of the theorem.
\end{proof}

\section{Multi-sheeted reflections via isothermal coordinates }

Let 
$f=(f_1,f_2,f_3)):Y_0 \rightarrow {\bf
R}^n$  be as in Theorem~\ref{isothermal}.  The equality
\begin{equation} \label{match0}
f_i(x,0)=f_j(x,0)
\end{equation}
for $i,j=1,2,3$ implies 
\begin{equation} \label{match}
\frac{\partial f_i}{\partial x}(x,0)=\frac{\partial
f_j}{\partial x}(x,0).
\end{equation}
 Using the conformality of $f_i$,  (\ref{geombal}) can be written as
\[
0=\sum_{i=1}^3  \frac{\frac{\partial f_i}{\partial
y}}{|\frac{\partial f_i}{\partial y}|} (x,0) =\sum_{i=1}^3
\frac{\frac{\partial f_i}{\partial y}}{|\frac{\partial
f_i}{\partial x}|} (x,0).
\]
 This combined  with with
(\ref{match})  implies
\begin{equation} \label{balancing}
0=\sum_{i=1}^3  \frac{\partial f_i}{\partial y}(x,0).
\end{equation}
If we let
\begin{equation} \label{oddref}
\tilde{f}_1(x,y)=-f_1(x,-y)+\frac{2}{3}
\sum_{i=1}^3 f_i(x,-y)
\end{equation}
then  (\ref{match0})  and (\ref{balancing}) imply that
\begin{equation} \label{setup}
f_1(x,0)=\tilde{f}_1(x,0) \ \mbox{ and } \
\frac{\partial f_1}{\partial y}(x,0)=\frac{\partial
\tilde{f}_1}{\partial y}(x,0).
\end{equation}
We claim (\ref{setup}) shows $U_1:\triangle \rightarrow {\bf R}^n$
defined by setting
\[
U_1(x,y) = \left\{
\begin{array}{ll}
f_1(x,y) & \mbox{for } y \geq 0\\
\tilde{f}_1(x,y)  & \mbox{for } y<0
\end{array}
\right.
\]
is harmonic.  Indeed, for any smooth $\xi:\triangle \rightarrow
{\bf R}^n$ with compact support, integration by parts gives
\[
-\int_{\triangle^+} \nabla \xi \cdot \nabla U_1 dxdy =
\int_{\triangle^+} \xi \triangle f_1 dxdy - \int_I \xi
\frac{\partial f_1}{\partial y}(x,0) dx
\]
and
\[
-\int_{\triangle^-} \nabla \xi \cdot \nabla U_1 dxdy =
\int_{\triangle^-} \xi \triangle \tilde{f}_1 dxdy + \int_I
\xi \frac{\partial \tilde{f}_1}{\partial y}(x,0) dt
\]
where $\triangle^+=\{ (x,y) \in \triangle: y>0\}$, $\triangle^-=\{(x,y) \in \triangle: y<0\}$ and $I=\{(x,y) \in \partial \triangle^+: y=0\}$.   Summing up the above two equations and using the
harmonicity of $f_1$ and $\tilde{f}_1$, we obtain
\[
-\int_{\triangle}  \nabla \xi \cdot \nabla U_1 dxdy =0.
\]
By Weyl's Lemma, $U_1$ is a $C^{\omega}$ harmonic map.   Similarly, there exists $C^{\omega}$ extensions $U_2,U_3$ of $f_2$ and $f_3$.  
We call this construction of the real analytic extension $U_i$ of $f_i$ the {\it multi-sheeted reflection.}

By summarizing the argument above, we have

\begin{theorem} \label{reflection}
Let $\Sigma_1$, $\Sigma_2$, $\Sigma_3$ and $\gamma$ as in Theorem~\ref{three}.   The surface $\Sigma_i$ can be extended real analytically across the curve $\gamma$.  This extended surface is parametrized by the conformal, harmonic map  $U_i$ via the multi-sheeted reflection. 
\end{theorem}

We note that the extendability of the minimal surface $\Sigma_i$ across a real analytic boundary curve  $\gamma$  follows from  a celebrated 
result of H.Lewy~\cite{lewy}.  On the other hand, Theorem~\ref{reflection} gives a more precise picture of the extension.    Indeed, the extension of the parameterization $f_1$  of $\Sigma_1$ is given in terms of a linear combination of odd reflections of $f_1,f_2,f_3$ as defined in (\ref{oddref}).

When three minimal surfaces are geometrically balanced along a  curve in ${\bf R}^3$ (i.e. the unit outer normal of the three surfaces sum to zero as in (\ref{geombal})), 
the entire configuration is completely determined by one of the three surfaces. This follows from the so-called  Bj\"{o}ling's problem resolved by H.Schwarz.   We will explain below how to use this and arguments in the proof of  Theorem~\ref{reflection} to give a  construction of Lewy's extension.  

 \begin{theorem} \label{lewy'}
 A minimal surface in ${\bf R}^3$ with a real analytic boundary can be extended across the boundary
 by a multi-sheeted reflection.
 \end{theorem}
 
 \begin{proof}
We start with a surface $\Sigma_1$ with a real analytic boundary curve $\gamma$.  
Let $\eta_1$  be the unit outer normal to the surface $\Sigma_1$ along $\gamma$.  
 Let $\eta_2$ and $\eta_3$ be the two unit vector fields 
defined on $\gamma$, normal to $\gamma$, each making the angle of $\pi/3$ to $\eta_1$,
Note here $\eta_1 + \eta_2 + \eta_3 = 0$.  
The solution by Schwarz of the Bj\"{o}rlng's problem (\cite{nitsche} III \S 149)  then provides locally defined, uniquely determined, minimal surfaces $\Sigma_2$ and $\Sigma_3$ along $\gamma$ so that $\eta_2$ and $\eta_3$ are unit outer normals to $\Sigma_2$ and $\Sigma_3$ along $\gamma$ respectively.

The next step is to construct  real analytic local parameterizations of the singular surface $(\cup_{i=1}^3 \Sigma_i) \cup \gamma$ by a continuous map from $Y_0$; i.e.  we construct maps from $\triangle^+_i \cup I_i$ to $\Sigma_i \cup \gamma$ for $i=1,2,3$  so that their restrictions to $I_i$ agree.    For example, one can use the hodographic projection
of ~\cite{KNS} to devise maps satisfying the required condition.     We briefly recall the hodographic projection. 
The surfaces $\Sigma_2$ and $\Sigma_3$ are locally graphs over the tangent plane
$T_q \Sigma_1$.   We suppose that $\Sigma_3$ lies above the plane, $\Sigma_2$ below. 
Let $\Pi_2$ and $\Pi_3$ be the orthogonal projection maps from $\Sigma_2 \cap B_\delta(q)$
and $\Sigma_3 \cap B_\delta(q)$ to $T_q \Sigma_1$ for sufficiently small $\delta$.  We can choose the coordinates so that $q=(0, 0, 0) \in {\bf R}^3$, the tangent line to $\gamma$ at $q$ is the $x_1$-axis and  $T_q \Sigma_1$ is the $x^1x^2$-plane.  Define $u_2,u_3$ by the conditions $\Pi^{-1}_2(x_1,x_2) = (x_1,x_2, u_2(x)) \in \Sigma_2 $ and $\Pi^{-1}_3(x_1,x_2) =(x_1,x_2, u_3(x))\in \Sigma_3$.  Near the origin, the map $h(x_1, x_2) := (x_1, u_3(x)-u_2(x))$ is 
of rank two. It sends $\Pi_2(\gamma) = \Pi_3(\gamma)$ to the $x_1$-axis and its image is contained in the upper half plane.  The map $h \circ \Pi_2$ is the hodographic projection and its inverse map $\Pi_2^{-1} \circ h^{-1}$ defined on a
sufficiently small half disk centered at the origin defines the real analytic parameterization of $\Sigma_2$
around $q$, while $\Pi_3^{-1} \circ h^{-1}$ defines that of $\Sigma_3$.    Similarly, $\Sigma_1$ can be parameterized using the tangent plane $T_q(\Sigma_2)$.  Denote these three maps parameterizing the neighborhood of $q$ by $H=(H_1, H_2, H_3):Y_0 \rightarrow (\cup_{i=1}^3 \Sigma_i) \cup \gamma$ with  $H_i :\triangle_i^+ \cup I_i  \rightarrow \Sigma_i$.  The real analyticity and the continuity of $H$ follows from the construction.  

For each $i=1,2,3$, the pulled-back metric $H_i^*G_0$ on $\triangle_i$ (recall that $G_0$ is the Euclidean metric on ${\bf R}^3$) defines a  conformal structure which 
in turn defines a Beltrami coefficient $\mu_i$.   This Beltrami coefficient is real analytic due to the argument in the proof of Theorem~\ref{isothermal}.     We again solve the Beltrami equation locally around a point $p_0=(x_0, 0) \in A_i \subset \overline{\triangle}_i $ via the Cauchy-Kowalewski theorem.  The solutions to the three Beltrami
equations give a parameterization $F=(F_1, F_2, F_3)$ of a neighborhood of $p_0$ in $X_{{\rm Id}}$ by $(X_{Id}, g_0)$,
an isothermal coordinate system.    Define $f=(f_1, f_2, f_3)$ by $f_i = g_i \circ F_i$.   Since $f_i:(\triangle^+,f^*_iG_0) \rightarrow {\bf R}^3$  is an isometric minimal immersion, $f_i$ is  harmonic with respect to the metric $f^*_iG_0$ (cf. \cite{La}).   As harmonicity is conformally invariant on two dimensional domains, $f_i$ is also harmonic with respect to the Euclidean metric on $\triangle^+_i$.  Hence the resulting parameterization $f_i:\triangle_i^+  \cup I_i \rightarrow (\Sigma_i \cup \gamma) \subset {\bf R}^3$ is
harmonic, conformal and without branch point.  Furthermore, we also have
\[
\sum_{i=1}^3\frac{\partial f_i}{\partial y} = c \sum_{i=1}^3 \eta_i =0
\]
where $c=\left| \frac{\partial f_i}{\partial x} \right|=\left| \frac{\partial f_i}{\partial y} \right|$. Now each $f_i: \triangle_i \rightarrow {\bf R}^3$
can be extended across the real axis $I_i$ by $\tilde{f}_1$ of (\ref{oddref}).  In particular, we have a conformal parameterization of the extension of  $\Sigma_1$ as in Theorem~\ref{reflection}.  
\end{proof}

\end{document}